\documentclass[11pt]{article}
\usepackage{amssymb}
\usepackage{amsfonts}
\usepackage{amsmath}
\usepackage{graphicx}

\newtheorem{theorem}{Theorem}
\newtheorem{acknowledgement}[theorem]{Acknowledgement}

\newtheorem{definition}[theorem]{Definition}

\newtheorem{lemma}[theorem]{Lemma}

\newtheorem{proposition}[theorem]{Proposition}
\newtheorem{remark}[theorem]{Remark}

\newenvironment{proof}[1][Proof]{\noindent\textbf{#1.} }{\ \rule{0.5em}{0.5em}}
\input{tcilatex}
\begin{document}

\title{On the existence of finite critical trajectories in a family of
quadratic differentials}
\author{Faouzi Thabet \\
The University of Gab\`{e}s. Tunisia.}
\maketitle
\title{}

\begin{abstract}
In this note, we discuss the possible existence of finite critical
trajectories connecting two zeros in a family of quadratic differentials
satisfying some properties. We treat the cases of holomorphic and
meromorphic quadratic differentials. In addition, we reprove some results
about the support of limiting root-counting measures of the generalized
Laguerre and Jacobi polynomials with varying parameters.
\end{abstract}



\section{\protect\bigskip Introduction and main results\label{intro}}

Quadratic differentials appear in many areas of mathematics and mathematical
physics such as orthogonal polynomials, moduli spaces of algebraic curves,
univalent functions, asymptotic theory of linear ordinary differential
equations etc...

One of the most common problems in the study of a given quadratic
differential, is the existence or not of its short trajectories. In this
note, we answer this question, under suitable assumptions.

In section \ref{app}, we present new proofs of the existence of short
trajectories of quadratic differentials related to generalized Laguerre and
Jacobi polynomials with varying parameters.

Let $\Omega $ be a non empty connected subset of $%
\mathbb{C}
,$ and $Q\left( z\right) =\prod_{k=1}^{3}\left( z-a_{k}\right) ^{m_{k}}$ be
a polynomial that is with simple or double zeros ($m_{k}\in \left\{
1,2\right\} $). Let $a,b:\Omega \longrightarrow $ $%
\mathbb{C}
\setminus \left\{ a_{1},a_{2},a_{3}\right\} $ be two continuous functions
such that
\begin{equation}
\forall t\in \Omega ,a\left( t\right) \neq b\left( t\right) .  \label{1}
\end{equation}%
We consider the family of rational and polynomial functions $R_{t}$ and $%
P_{t}$
\begin{eqnarray*}
R_{t}\left( z\right) &=&\frac{\left( z-a\left( t\right) \right) \left(
z-b\left( t\right) \right) }{Q\left( z\right) }, \\
P_{t}\left( z\right) &=&\left( z-a\left( t\right) \right) \left( z-b\left(
t\right) \right) Q\left( z\right) .
\end{eqnarray*}%
We denote $\mathcal{J}_{a\left( t\right) ,b\left( t\right) }$ the set of all
Jordan arcs in $%
\mathbb{C}
\setminus \left\{ a_{1},a_{2},a_{3}\right\} $ joining $a\left( t\right) $
and $b\left( t\right) ,$ and we suppose that there exists a continuous
function (in the Haussd\H{o}rf metric)
\begin{equation*}
\begin{array}{cc}
\Phi :\Omega \longrightarrow \mathcal{J}_{a\left( t\right) ,b\left( t\right)
} & t\longmapsto \phi _{t},%
\end{array}%
\end{equation*}%
such that
\begin{equation}
\phi _{t}\left( 0\right) =a\left( t\right) ,\phi _{t}\left( 1\right)
=b\left( t\right) .  \label{2}
\end{equation}%
We assume that for some choice of branches of the square roots $\sqrt{\
R_{t}\left( z\right) }$ and $\sqrt{\ P_{t}\left( z\right) },\phi _{t}$
satisfies
\begin{eqnarray}
\Re \int_{\phi _{t}}\sqrt{\ R_{t}\left( z\right) }dz &=&0;  \label{3} \\
\Re \int_{\phi _{t}}\sqrt{\ P_{t}\left( z\right) }dz &=&0.  \label{4}
\end{eqnarray}%
We consider the quadratic differentials
\begin{eqnarray*}
\varpi \left( R_{t},z\right) &=&-R_{t}\left( z\right) dz^{2}, \\
\varpi \left( P_{t},z\right) &=&-P_{t}\left( z\right) dz^{2}.
\end{eqnarray*}%
Then, the following results hold

\begin{proposition}
\label{rational}Under assumptions (\ref{1}),(\ref{2}), and (\ref{3}),
either, for any $t\in \Omega ,$ there exists exactly one short trajectory of
the quadratic differential $\varpi \left( R_{t},z\right) $ that connects $%
a\left( t\right) $ and $b\left( t\right) ,$ homotopic to $\phi _{t}$ in $%
\mathbb{C}
\setminus \left\{ a_{1},a_{2},a_{3}\right\} ,$ or there does not exist any
such trajectory for any $t\in \Omega .$
\end{proposition}

\begin{proposition}
\label{polyn}With assumptions (\ref{1}),(\ref{2}), and (\ref{4}), the set of
all $t\in \Omega $ such that $\varpi \left( P_{t},z\right) $ has a short
trajectory connecting $a\left( t\right) $ and $b\left( t\right) $ is a
closed subset of $\Omega .$
\end{proposition}

\section{\protect\bigskip Basics of quadratic differentials}

We first present some basics for quadratic differentials.

\begin{definition}
A rational quadratic differential on the Riemann sphere $\overline{%
\mathbb{C}
}$ is a form $\varpi =\varphi (z)dz^{2}$, where $\varphi $ is a rational
function of a local coordinate $z$. If $z=z(\zeta )$ is a conformal change
of variables then
\begin{equation*}
\widetilde{\varphi }(\zeta )d\zeta ^{2}=\varphi (z(\zeta ))(dz/d\zeta
)^{2}d\zeta ^{2}
\end{equation*}%
represents $\varpi $ in the local parameter $\zeta $.
\end{definition}

The \emph{critical points} of $\varpi $ are its zeros and poles; a critical
point is \emph{finite} if it is a zero or a simple pole, otherwise, it is
\emph{infinite. }All other points of $\overline{%
\mathbb{C}
}$ are called \emph{regular} points.

The horizontal trajectories (or just trajectories ) are the zero loci of the
equation%
\begin{equation}
\Im \int^{z}\sqrt{\varphi \left( t\right) }dt=\text{\emph{const}},
\label{traj}
\end{equation}%
or equivalently%
\begin{equation*}
\varphi \left( z\right) dz^{2}>0;
\end{equation*}

the vertical trajectories are obtained by replacing $\Im $ by $\Re $ in the
equation above. The horizontal and vertical trajectories of $\varpi $
produce two pairwise orthogonal foliations of the Riemann sphere $\overline{%
\mathbb{C}
}$. A critical trajectory is a trajectory passing through a critical point.
A finite critical trajectory or \emph{short trajectory} is a critical
trajectory connecting two finite critical points of $\varpi $, it will be
called \emph{unbroken} if it is not passing through other finite critical
points except its two endpoints, otherwise, we call it \emph{broken}. The
set of finite and infinite critical trajectories of $\varpi $ together with
their limit points (critical points of $\varpi $) is called the \emph{%
critical graph} of $\varpi $.

Notice that, if $z\left( t\right) ,t\in
\mathbb{R}
$ is a trajectory of (\ref{traj}), then the function
\begin{equation*}
t\longmapsto \Im \int^{t}\sqrt{\varphi \left( z\left( u\right) \right) }%
z^{\prime }\left( u\right) du
\end{equation*}%
is monotone. For more details, we refer the reader to \cite{Striebel}.

The local structure of the trajectories is as follow :

\begin{itemize}
\item At any regular point horizontal (resp. vertical) trajectories look
locally as simple analytic arcs passing through this point, and through
every regular point of $\varpi $ passes a uniquely determined horizontal
(resp. vertical) trajectory of $\varpi ;$ these horizontal and vertical
trajectories are locally orthogonal at this point.

\item From every zero with multiplicity $r$ of $\varpi $, there emanate $%
\left( r+2\right) $ horizontal (resp. vertical) trajectories, and the angle
between any two adjacent trajectories equals $\pi /\left( r+2\right) .$

\item At a simple pole there emanates only one trajectory (see Figure \ref{Figa}).

\item At a double pole, the local behavior of the trajectories depends on
the vanishing of the real or imaginary part of the residue; they have either
the radial, the circular or the log-spiral form (Figure \ref{Figb}).

\item At a pole of order $r$ greater than $2,$ there are $\left( r-2\right) $
asymptotic directions (called \emph{critical directions}) spacing with equal
angle $\frac{2\pi }{r-2},$ and a neighborhood $\mathcal{U}$, such that each
trajectory entering $\mathcal{U}$ stays in $\mathcal{U}$ and tends to this
pole in one of the critical directions.
\end{itemize}
\begin{figure}[h]%
\centering
\fbox{\includegraphics[
height=2.0903in,
width=3.7048in
]%
{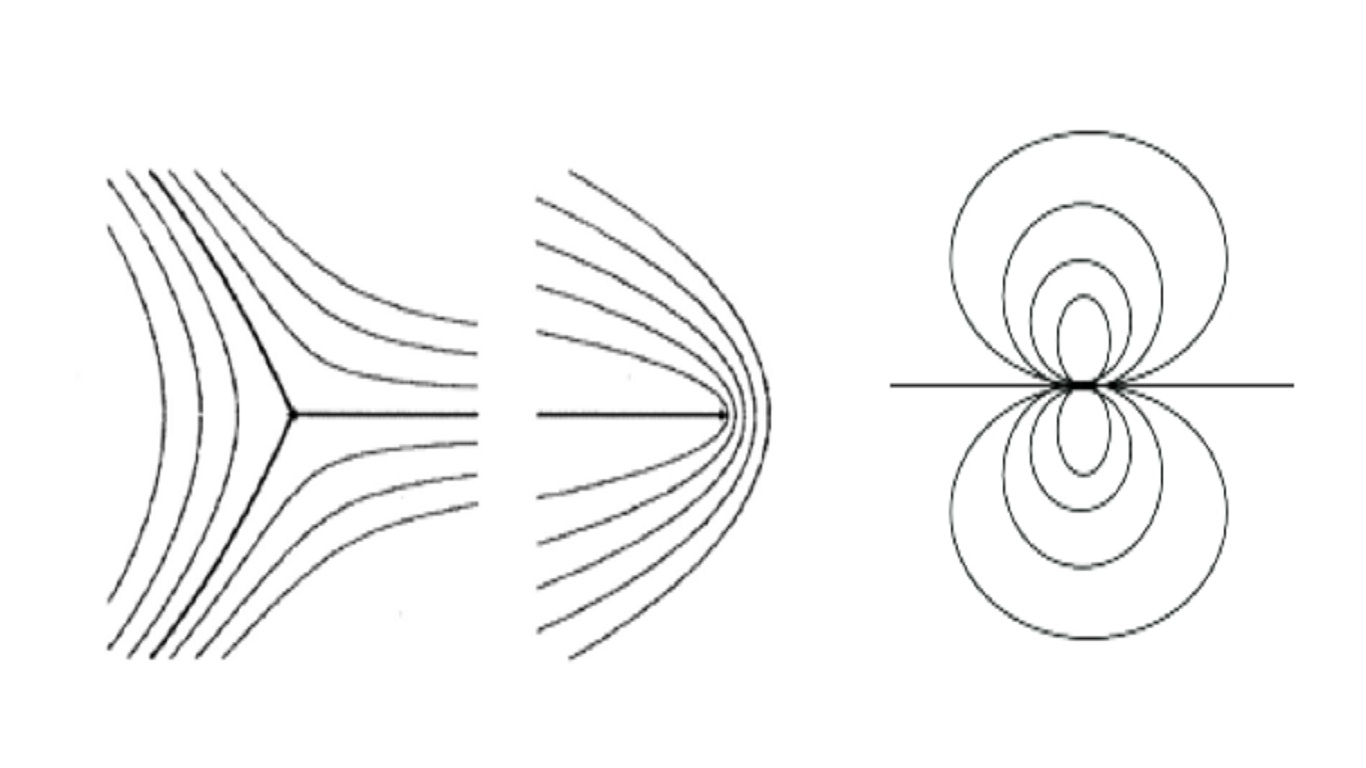}
}\caption{Strucure of trajectories in a
neighborhood of a simple zero (left), simple pole (middle), and a 4th order
pole (right)..}%
\label{Figa}%
\end{figure}
\begin{figure}[h]%
\centering
\fbox{\includegraphics[
height=2.0903in,
width=3.7048in
]%
{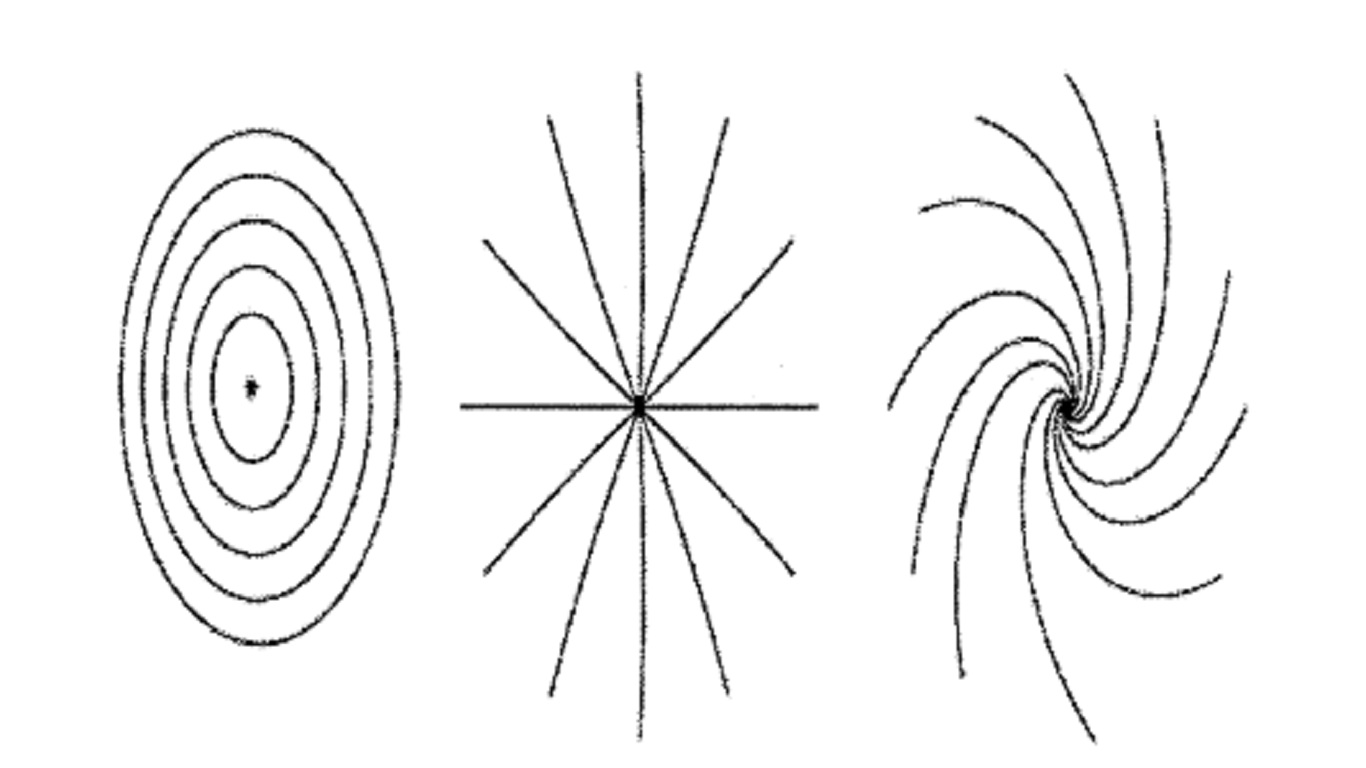}%
}\caption{Structure of trajectories in a
neighborhood of a double pole, circle form (left), radial form (middle), and
log-spiral form (right).}
\label{Figb}%
\end{figure}

\bigskip The main trouble in the global behaviour of trajectories which are
dense in some domains in $%
\mathbb{C}
$ comes from the so-called recurrent trajectory; Jenkins' three pole Theorem
asserts that such a situation cannot happen for a quadratic differential
that has at most three poles.

A necessary condition for the existence of a short trajectory connecting two
finite critical points $a$ and $b$ of a quadratic differential $\varphi
\left( z\right) dz^{2}$ is the existence of Jordan arc $\gamma $ connecting $%
a$ and $b$ in $%
\mathbb{C}
\setminus \left\{ \text{poles of }\varphi \right\} ,$ such that
\begin{equation*}
\Im \int_{\gamma }\sqrt{\varphi \left( t\right) }dt=0,
\end{equation*}%
but this condition is not sufficient. Indeed, Figure \ref{Figc}
illustrates the critical graph of the quadratic differential $Q\left(
z\right) =-\left( z^{4}-1\right) dz^{2}$; in particular, there is no short
trajectory connecting the zeros $\pm i.$ However, if $\gamma $ is an
oriented Jordan arc joining $\pm i$ in $%
\mathbb{C}
\setminus \left[ -1,1\right] $ , and $\sqrt{z^{4}-1}$ is chosen in $%
\mathbb{C}
\setminus \left( \left[ -1,1\right] \cup \gamma \right) $ with condition $%
\sqrt{z^{4}-1}$ $\backsim z^{2},z\rightarrow \infty ,$ then from the Laurent
expansion at $\infty $ of $\sqrt{z^{4}-1}:$%
\begin{equation*}
\sqrt{z^{4}-1}=z^{2}+\mathcal{\allowbreak O}\left( z^{-2}\right)
,z\rightarrow \infty ,
\end{equation*}%
we deduce the residue of $\sqrt{z^{4}-1}$ at $\infty $ :
\begin{equation*}
res_{\infty }\left( \sqrt{z^{4}-1}\right) =\allowbreak 0.
\end{equation*}%
For $t\in \left[ -1,1\right] \cup \gamma ,$ we denote by $\left( \sqrt{%
t^{4}-1}\right) _{+}$ and $\left( \sqrt{t^{4}-1}\right) _{-}$ the limits
from the $+$-side and $-$-side respectively. (As usual, the $+$-side of an
oriented curve lies to the left, and the $-$-side lies to the right, if one
traverses the curve according to its orientation.)

Let%
\begin{equation*}
I=\int_{-1}^{1}\left( \sqrt{t^{4}-1}\right) _{+}dt+\int_{\gamma }\left(
\sqrt{t^{4}-1}\right) _{+}dt.
\end{equation*}%
Since $\left( \sqrt{t^{4}-1}\right) _{+}=-\left( \sqrt{t^{4}-1}\right) _{-},$
for $t\in \left[ -1,1\right] \cup \gamma ,$ we have
\begin{equation*}
2I=\int_{\left[ -1,1\right] \cup \gamma }\left[ \left( \sqrt{t^{4}-1}\right)
_{+}-\left( \sqrt{t^{4}-1}\right) _{-}\right] dt=\oint_{\Gamma _{i,j}\cup
\Gamma _{l,k}}\sqrt{z^{4}-1}dz,
\end{equation*}%
where $\Gamma _{i,j}$ and $\Gamma _{l,k}$ are two closed contours encircling
respectively the curve $\left[ -1,1\right] $ and $\gamma $ once in the
clockwise direction. After the contour deformation we pick up residue at $%
z=\infty .$ We get
\begin{equation*}
I=\frac{1}{2}\oint_{\Gamma _{i,j}\cup \Gamma _{l,k}}\sqrt{z^{4}-1}dz=\pm
i\pi res_{\infty }\left( \sqrt{z^{4}-1}\right) =0.
\end{equation*}%
By the other hand, it is straightforward that $\Re \int_{-1}^{1}\left( \sqrt{%
t^{4}-1}\right) _{+}dt=0,$ which implies that
\begin{equation*}
\Re \int_{\gamma }\left( \sqrt{t^{4}-1}\right) _{+}dt=0.
\end{equation*}%
\begin{figure}[h]%
\centering
\fbox{\includegraphics[
height=2.0903in,
width=3.7048in
]%
{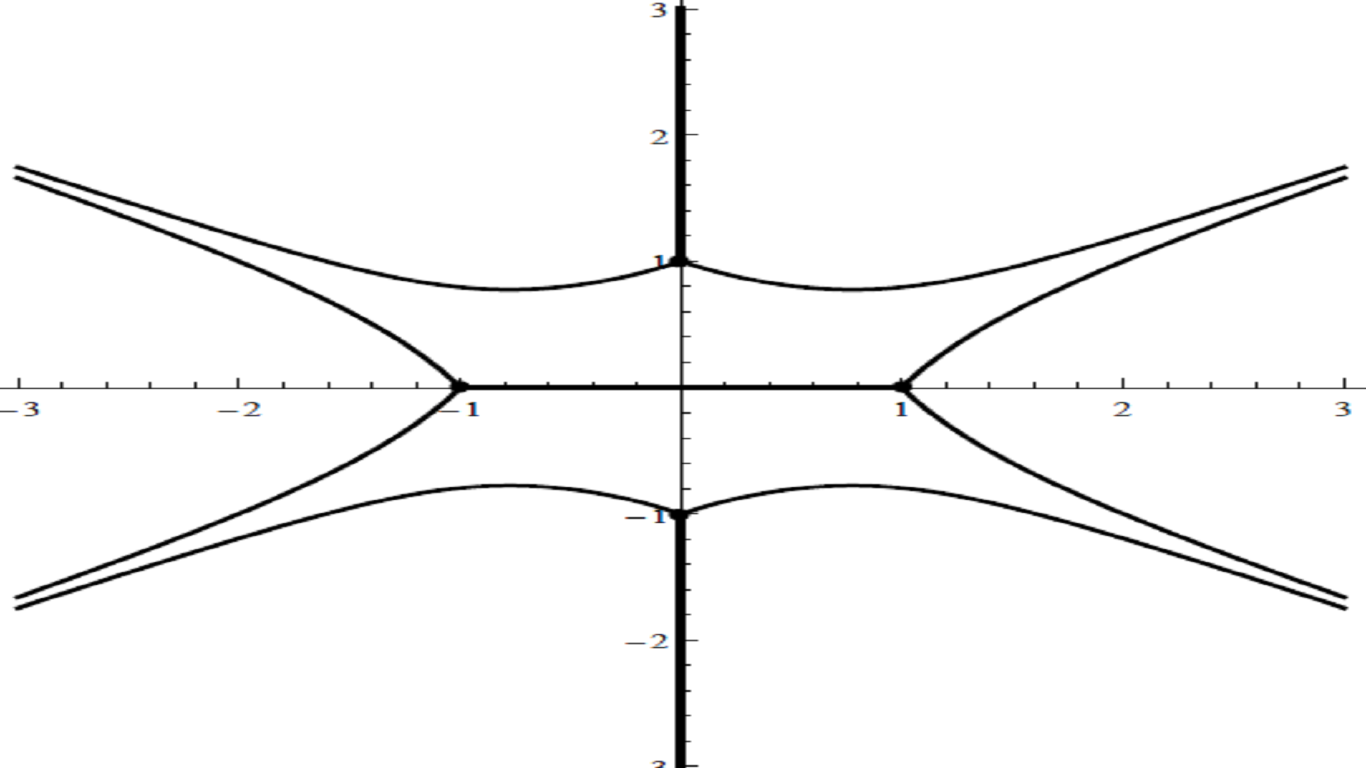}%
}\caption{ Critical graph of the quadratic
differential $Q\left( z\right) =-\left( z^{4}-1\right) dz^{2}.$}
\label{Figc}%
\end{figure}

\bigskip The quadratic differential $\varphi \left( z\right) dz^{2}$ defines
a $\varphi $-metric with the differential element $\sqrt{\left\vert \varphi
\left( z\right) \right\vert }\left\vert dz\right\vert $. If $\gamma $ is a
rectifiable arc in $\overline{%
\mathbb{C}
}$, then its $\varphi $-length is defined by%
\begin{equation*}
\left\vert \gamma \right\vert _{\varphi }=\int_{\gamma }\sqrt{\left\vert
\varphi \left( z\right) \right\vert }\left\vert dz\right\vert .
\end{equation*}

A trajectory of $\varphi \left( z\right) dz^{2}$ is finite if, and only if
its $\varphi $-length is finite, otherwise is infinite. In particular, a
critical trajectory is finite if and only if its two end points each are
finite critical point.

Two Jordan arcs $\alpha ,\beta :\left[ 0,1\right] \longrightarrow
\mathbb{C}
$ joining a point $p_{1}$ to a point $p_{2}$ in $%
\mathbb{C}
\setminus \left\{ \text{poles of }\varphi \right\} $ are homotopic if there
exists a continuous function $H:\left[ 0,1\right] \times $ $\left[ 0,1\right]
\longrightarrow $ $%
\mathbb{C}
\setminus \left\{ \text{poles of }\varphi \right\} $ such that%
\begin{equation*}
\left\{
\begin{array}{c}
H\left( t,0\right) =\alpha \left( t\right) \\
H\left( t,1\right) =\beta \left( t\right)%
\end{array}%
,t\in \left[ 0,1\right] .\right.
\end{equation*}%
It is an equivalence relation on the set $\mathcal{J}_{p_{1},p_{2}}$ of all
Jordan arcs joining $p_{1}$ to $p_{2}$ in $%
\mathbb{C}
\setminus \left\{ \text{poles of }\varphi \right\} $. If $card\left\{ \text{%
poles of }\varphi \right\} =m\in
\mathbb{N}
,$ then, it is well known that $%
\mathbb{C}
\setminus \left\{ \text{poles of }\varphi \right\} $ and the wedge of $m$
circles have the same type of homotopy; in particular, there are $2^{m}$
classes of equivalence of the relation "homotopic" on $\mathcal{J}%
_{p_{1},p_{2}}$.

\begin{definition}
A locally rectifiable (in the spherical metric) curve $\gamma _{0}$ is
called a $\varphi $-geodesic if it is locally shortest in the $\varphi $%
-metric. It is called critical geodesic it is $\varphi $-geodesic and
passing through a critical point of the quadratic differential $\varphi
\left( z\right) dz^{2}.$
\end{definition}

\begin{proposition}[{\protect\cite[Theorem 16.2]{Striebel}}]
\label{geod}\bigskip Let $\gamma $ be a $\varphi $-geodesic arc joining $%
p_{1}$ to $p_{2}$ in $%
\mathbb{C}
\setminus \left\{ \text{poles of }\varphi \right\} .$ Then for every $\gamma
_{1}\in $ $\mathcal{J}_{p_{1},p_{2}}$ which is homotopic to $\gamma $ on $%
\mathbb{C}
\setminus \left\{ \text{poles of }\varphi \right\} $, we have $\left\vert
\gamma _{1}\right\vert _{\varphi }\geq $ $\left\vert \gamma \right\vert
_{\varphi },$ with equality if and only if $\gamma _{1}=\gamma .$
\end{proposition}

We finish this section by the so-called Teichmuller Lemma that will be in
use in the next section.

\begin{definition}
\bigskip A domain in $%
\mathbb{C}
$ bounded only by segments of $\varphi $-geodesic and/or horizontal and/or
vertical trajectories of the quadratic differential $\varphi \left( z\right)
dz^{2}$ (and their endpoints) is called $\varphi $-polygon.
\end{definition}

\begin{lemma}[Teichmuller]
\label{teichmuller}Let $\Omega $ be a $\varphi $-polygon, and let $z_{j}$ be
the singular points of $\varphi \left( z\right) dz^{2}$ on the boundary $%
\partial \Omega $ of $\Omega ,$ with multiplicities $n_{j},$ and let $\theta
_{j}$ $\in \left[ 0,2\pi \right] $be the corresponding interior angles with
vertices at $z_{j},$ respectively. Then%
\begin{equation}
\sum \left( 1-\theta _{j}\dfrac{n_{j}+2}{2\pi }\right) =2+\sum n_{i},
\label{teich}
\end{equation}%
where $n_{i}$ are the multiplicities of the singular points inside $\Omega .$
\end{lemma}

\section{Proofs}

\begin{lemma}
In the notation of Proposition \ref{2}

\begin{enumerate}
\item[(a)] there exists at most one unbroken short trajectory of the
quadratic differential $\varpi \left( P_{t},z\right) $ connecting $a\left(
t\right) $ and $b\left( t\right) $.

\item[(b)] If there exist two short trajectories of the quadratic
differential $\varpi \left( R_{t},z\right) $ connecting $a\left( t\right) $
and $b\left( t\right) ,$ then they are not homotopic in $%
\mathbb{C}
\setminus \left\{ a_{1},a_{2},a_{3}\right\} $.
\end{enumerate}
\end{lemma}

\begin{proof}

\begin{enumerate}
\item[(a)] Suppose that $\gamma _{1}$ and $\gamma _{2}$ are two unbroken
short trajectories of $\varpi \left( P_{t},z\right) $ connecting $a\left(
t\right) $ and $b\left( t\right) ,$ and let $\Omega $ be the $\varpi $%
-polygon with vertices $a\left( t\right) $ and $b\left( t\right) ,$ and
edges $\gamma _{1}$ and $\gamma _{2}$. From Lemma \ref{teichmuller}, the
left-hand side of (\ref{teich}) is smaller than $2$, whereas the righthand
side is clearly at least $2$, a contradiction.

\item[(b)] In the same vein of the previous proof, by taking $\gamma _{1}$
and $\gamma _{2}$ are two short trajectories of $\varpi \left(
R_{t},z\right) $ connecting $a\left( t\right) $ and $b\left( t\right) ;$ the
fact that are homotopic in $%
\mathbb{C}
\setminus \left\{ a_{1},a_{2},a_{3}\right\} $ means that there is no pole of
$R_{t}$ inside $\Omega ;$ and again, we get a contradiction with Lemma \ref%
{teichmuller}.
\end{enumerate}
\end{proof}

\begin{remark}
The number of unbroken short geodesics of $\varpi _{P}\left( t,z\right) $
can be any integer between $\deg \left( P_{t}(z)\right) -1$ and $\left(
\begin{array}{c}
2 \\
\deg \left( P_{t}(z)\right)%
\end{array}%
\right) $. We refer the reader to \cite{SHA} for the proof.
\end{remark}

\begin{remark}
It is well known that, by using $3$ wedged circles, there are $8$ homotopy
classes in $\mathcal{J}_{a\left( t\right) ,b\left( t\right) }.$ With the
same way of the previous proof, and by Proposition \ref{geod}, there exist
at most $8$ unbroken short geodesics of $\varpi \left( R_{t},z\right) $
joining $a\left( t\right) $ and $b\left( t\right) .$
\end{remark}

\begin{proof}[Proof of Proposition \protect\ref{rational}]
Let us denote $\Lambda $ the subset of $\Omega $ formed by all $t$ such that
there exists a short trajectory of $\varpi \left( R_{t},z\right) $ homotopic
to $\phi _{t}$ in $%
\mathbb{C}
\setminus \left\{ a_{1},a_{2},a_{3}\right\} .$

Let $t_{0}$ $\in \Lambda .$ By continuity of the quadratic differential $%
\varpi \left( R_{t},z\right) $, for every $\varepsilon >0$ there exists $%
\delta >0$ such that for any $t\in \Omega $ satisfying $|t-t_{0}|<\delta $,
there exists a trajectory of $\varpi \left( R_{t},z\right) ,$ say $\gamma
_{t},$ emanating from $a\left( t\right) $ and intersecting the $\varepsilon $%
-neighborhood $\mathcal{U}_{\varepsilon }$ of $b(t)$. If $\gamma _{t}$ does
not pass through $b(t)$, then, we may assume that $\delta >0$ is small
enough so that $\gamma _{t}$ is intersected by an orthogonal trajectory $%
\sigma _{t}$ emanating from $b(t)$ in some point $c\left( t\right) $. We
denote by $\varphi _{t}$ the path that follows the arc of $\gamma _{t}$ from
$a(t)$ to $c\left( t\right) $ and then continues to $b(t)$ along $\sigma
_{t} $. Clearly the arcs $\phi _{t}$ and $\varphi _{t}$ are homotopic in $%
\mathbb{C}
\setminus \left\{ a_{1},a_{2},a_{3}\right\} ,$ and, by definition of
orthogonal trajectories, the real part of the integral along $\varphi _{t}$
of $\sqrt{\ R_{t}\left( z\right) }$ cannot vanish. This contradiction shows
that the whole small neighborhood of $t_{0}$ still in $\Lambda ,$ and then $%
\Lambda $ is an open subset of $\Omega $.

Suppose now that $\left( t_{n}\right) $ is sequence of $\Lambda $ converging
to $t\in \Omega $, so that $a\left( t_{n}\right) $ and $b\left( t_{n}\right)
$ converge respectively to $a$ and $b.$ For each $t_{n}$, there exists a
unique short trajectory $\gamma _{n}$ joining $a\left( t_{n}\right) $ and $%
b\left( t_{n}\right) ,$ and all the $\gamma _{n}$ are homotopic to $\phi
_{t_{n}}$ in $%
\mathbb{C}
\setminus \left\{ a_{1},a_{2},a_{3}\right\} $. It is obvious that the limit
set of the sequence $\gamma _{n}$ (in the Hausd\H{o}rff metrics) is either
another short trajectory connecting $a$ and $b$, or a union of two infinite
critical trajectories $\gamma _{a}$ and $\gamma _{b}$ emanating respectively
from $a$ and $b,$ and each of them diverges to some pole of the quadratic
differential $\varpi \left( R_{t},z\right) .$ If $\gamma _{a}$ and $\gamma
_{b}$ do not diverge to the same pole, or one of them diverges to a simple
pole, then
\begin{equation*}
\inf_{x\in \gamma _{a},y\in \gamma _{b}}\left\vert x-y\right\vert
=dist\left( \gamma _{a},\gamma _{b}\right) >0,
\end{equation*}%
which contradicts the fact that $\lim_{n\rightarrow \infty }\gamma
_{n}=\gamma _{a}$ $\cup \gamma _{b}.$ Let $c$ $\in \left\{
a_{1},a_{2},a_{3}\right\} \cup \left\{ \infty \right\} $ be the common pole
of divergence of $\gamma _{a}$ and $\gamma _{b}.$

If $c$ is a double pole (we assume without loss of generality that the
residue of the quadratic differential $\varpi \left( R_{t},z\right) $ at the
pole $c$ is non real, and then $\gamma _{a}$ and $\gamma _{b}$ diverge to $c$
in log-spiral). Let $\sigma $ be an orthogonal trajectory that diverges (of
course, in log-spiral) to $c$. Then $\sigma $ intersects $\gamma _{a}$ and $%
\gamma _{b}$ infinitely many times. Considering three consecutive points of
intersection, it is obvious that we can construct two paths $\gamma $ and $%
\gamma ^{\prime }$ joining $a$ and $b$ formed by the three parts, from $%
\gamma _{a},\sigma $ and $\gamma _{b}.$ Clearly, $\gamma $ and $\gamma
^{\prime }$ are not homotopic in $%
\mathbb{C}
\setminus \left\{ c\right\} ,$ and by continuity of the family $\phi
_{t_{n}},$ one of them must be homotopic to $\phi _{t_{n}}$ for $n\geq n_{0}$
for some integer $n_{0}$. Then we get
\begin{equation*}
\Re \int_{\gamma }\sqrt{\ R_{t}\left( z\right) }dz\neq 0,\text{ and }\Re
\int_{\gamma ^{\prime }}\sqrt{\ R_{t}\left( z\right) }dz\neq 0,
\end{equation*}%
which contradicts (\ref{3}). Then, the limit set of the sequence $\gamma
_{n} $ is a short trajectory joining $a\left( t\right) $ and $b\left(
t\right) ,$ and $\Lambda $ is a closed subset of $\Omega $. The cases when
the residues at $c$ are real (positive or negative) are in the same vein.

Finally, since $\Omega $ is a connected subset of $%
\mathbb{C}
,$ we conclude that either $\Lambda =\Omega ,$ or $\Lambda =\emptyset .$
\end{proof}

\begin{proof}[Proof of Proposition \protect\ref{polyn}]
In order to discuss the possible existence of a short trajectory of $\varpi
\left( P_{t},z\right) $ connecting $a\left( t\right) $ and $b\left( t\right)
$ for some $t\in \Omega ,$ we denote by $\Gamma _{a\left( t\right) }$ and $%
\Gamma _{b\left( t\right) }$ the sets of the three critical trajectories
that emanate respectively from $a\left( t\right) $ and $b\left( t\right) ,$
and we consider the euclidian distance
\begin{equation*}
dist\left( \Gamma _{a\left( t\right) },\Gamma _{b\left( t\right) }\right)
=\inf_{x\in \Gamma _{a\left( t\right) },y\in \Gamma _{b\left( t\right)
}}\left\vert x-y\right\vert .
\end{equation*}%
Then we claim the following : The quadratic differential $\varpi \left(
P_{t},z\right) $ has a short trajectory connecting $a\left( t\right) $ and $%
b\left( t\right) $, if and only if, $dist\left( \Gamma _{a\left( t\right)
},\Gamma _{b\left( t\right) }\right) =0.$ Indeed, there are $n+2$ asymptotic
directions, spacing with equal angle $\frac{2\pi }{n+2}$ that can take any
horizontal (resp. vertical) trajectory of the quadratic differential $\varpi
\left( P_{t},z\right) $ diverging to infinity; the asymptotic directions of
vertical trajectories are obtained by rotation of angle $\frac{\pi }{2}$.
Obviously, if $dist\left( \Gamma _{a\left( t\right) },\Gamma _{b\left(
t\right) }\right) >0,$ then there is no short trajectory connecting $a\left(
t\right) $ and $b\left( t\right) $. Assume that $dist\left( \Gamma _{a\left(
t\right) },\Gamma _{b\left( t\right) }\right) =0$ and no short trajectory
connects $a\left( t\right) $ and $b\left( t\right) $. Since $\Gamma
_{a\left( t\right) }\cap \Gamma _{b\left( t\right) }=\emptyset ,$ there
exist two horizontal trajectories $\gamma _{a\left( t\right) }$ and $\gamma
_{b\left( t\right) }$ that emanate from $a\left( t\right) $ and $b\left(
t\right) $ and diverge to infinity in the same direction $D$; let $\sigma $
be a vertical trajectory (not critical) diverging to infinity in the two
directions adjacent to $D.$ Obviously, $\sigma $ intersects $\gamma
_{a\left( t\right) }$ and $\gamma _{b\left( t\right) }$ in exactly two
points $P_{a\left( t\right) }$ and $P_{b\left( t\right) }.$ Let $\gamma \in
\mathcal{J}_{a\left( t\right) ,b\left( t\right) }$ be the union of the part
of $\gamma _{a\left( t\right) }$ from $a\left( t\right) $ to $P_{a\left(
t\right) },$ and the part of $\sigma $ from $P_{a\left( t\right) }$ to $%
P_{b\left( t\right) },$ and finally, the part of $\gamma _{b\left( t\right)
} $ from $P_{b\left( t\right) }$ to $b\left( t\right) $. Integrating along $%
\gamma ,$ and since
\begin{equation*}
\Re \int_{a\left( t\right) }^{P_{a\left( t\right) }}\sqrt{P_{t}\left(
z\right) }dz=\Re \int_{P_{b\left( t\right) }}^{b\left( t\right) }\sqrt{%
P_{t}\left( z\right) }dz=0,
\end{equation*}%
we get
\begin{equation*}
\Re \int_{\gamma }\sqrt{P_{t}\left( z\right) }dz=\Re \int_{P_{a\left(
t\right) }}^{P_{b\left( t\right) }}\sqrt{P_{t}\left( z\right) }dz\neq 0,
\end{equation*}%
which violates (\ref{4}). By continuity of the function $t\longmapsto $ $%
dist\left( \Gamma _{a\left( t\right) },\Gamma _{b\left( t\right) }\right) ,$
it follows that the set of all $t\in \Omega $ such that the quadratic
differential $\varpi \left( P_{t},z\right) $ has no short trajectory
connecting $a\left( t\right) $ and $b\left( t\right) ,$ is an open subset of
$\Omega $. Notice that Proposition \ref{polyn} still valid with polynomials $%
Q$ with higher degree or multiplicities of zeros.
\end{proof}

\bigskip

\section{\protect\bigskip Connection with Laguerre and Jacobi polynomials
\label{app}}

The rescaled generalized Laguerre polynomials $L_{n}^{nC}\left( nz\right) $
with varying parameters $nC,$ and the Jacobi polynomials $%
P_{n}^{(nA,nB)}\left( z\right) $ with varying parameters $nA$ and $nB$ can
be given explicitly respectively by (see \cite{Szego}):
\begin{equation*}
L_{n}^{nC}\left( nz\right) =\sum_{k=0}^{n}\left(
\begin{array}{c}
n+nC \\
n-k%
\end{array}%
\right) \frac{\left( -z\right) ^{k}}{k!},
\end{equation*}

\begin{equation*}
P_{n}^{(nA,nB)}\left( z\right) =2^{-n}\sum_{k=0}^{n}\left(
\begin{array}{c}
n+nA \\
n-k%
\end{array}%
\right) \left(
\begin{array}{c}
n+nB \\
\ k%
\end{array}%
\right) \left( z-1\right) ^{k}\left( z+1\right) ^{n-k}.
\end{equation*}%
Jacobi or Laguerre polynomials with (real) parameters, depending on the
degree $n$ appear naturally as polynomial solutions of hypergeometric
differential equations, or in the expressions of the wave functions of many
classical systems in quantum mechanics; see \cite{hyper}.

With each polynomial $p_{n}$, we associate its normalized zero-counting
measure $\mu _{n},$
\begin{equation*}
\mu _{n}=\mu \left( p_{n}\right) =\frac{\sum_{p_{n}\left( z\right) =0}\delta
_{z}}{n}.
\end{equation*}%
For a compact subset $K$ in $%
\mathbb{C}
,$
\begin{equation*}
\int_{K}d\mu _{n}=\frac{\text{number of zeros of }p_{n}\text{ in }K}{n}.
\end{equation*}%
The zeros are counted with their multiplicities.

Following the works of Gonchar-Rakhmanov \cite{gonchar} and Stahl \cite%
{stahl}, it was shown that the sequence $\mu _{n}$ converges (as $%
n\rightarrow \infty $) in the weak-* topology to a measure, supported on
short trajectories of related quadratic differentials. For the case of
Laguerre, see \cite{amf pgm ro},\cite{abj mac},\cite{Atia}; for the case of
Jacobi, see \cite{abjk amf},\cite{abjk amf ro},\cite{AMF FT},\cite%
{FTMChouikhi}.

The related quadratic differential for Laguerre polynomials is,
\begin{equation}
\varpi _{C}=-\frac{D_{C}(z)}{z^{2}}dz^{2},  \label{laguerre}
\end{equation}%
where%
\begin{equation*}
D_{C}(z)=z^{2}-2(C+2)z+C^{2}.
\end{equation*}%
The zeros of $D_{C}(z)$ are
\begin{equation}
a\left( C\right) =C+2+2\sqrt{C+1},b(C)=C+2-2\sqrt{C+1}.
\label{zeros laguerre}
\end{equation}%
The related quadratic differential for Jacobi polynomials is
\begin{equation}
\varpi _{A,B}=-\frac{D_{A,B}\left( z\right) }{\left( z^{2}-1\right) ^{2}}%
\,dz^{2},  \label{jacobi}
\end{equation}%
where
\begin{equation*}
D_{A,B}\left( z\right) =\left( A+B+2\right) ^{2}z^{2}+2\left(
A^{2}-B^{2}\right) z+\left( A-B\right) ^{2}-4\left( A+B+1\right) .
\end{equation*}%
The zeros of $D_{A,B}(z)$ are
\begin{eqnarray*}
a\left( A,B\right) &=&\frac{-A^{2}+B^{2}+4\sqrt{\left( A+1\right) \left(
B+1\right) \left( A+B+1\right) }}{\left( A+B+2\right) ^{2}}, \\
b(A,B) &=&\frac{-A^{2}+B^{2}-4\sqrt{\left( A+1\right) \left( B+1\right)
\left( A+B+1\right) }}{\left( A+B+2\right) ^{2}}.
\end{eqnarray*}

\begin{proposition}[\protect\cite{Atia}]
Assume that $C\in
\mathbb{C}
_{+}$, and that $\gamma $ is a Jordan arc connecting the zeros of $D_{C}(z)$
in the punctured plane $%
\mathbb{C}
\setminus \{0\}$. Denote by $\sqrt{D_{C}(z)}$ the single-valued branch of
this function in $%
\mathbb{C}
\setminus \gamma $ determined by the condition
\begin{equation*}
\sqrt{D_{C}(z)}\sim z,z\rightarrow \infty ,
\end{equation*}%
and let $\left( \sqrt{D_{C}(z)}\right) _{+}$ stand for its boundary values
on the $+$-side of $\gamma $. Then
\begin{equation}
\int_{\gamma }\frac{\left( \sqrt{D_{C}(t)}\right) _{+}}{t}\,dt\in \pm 2\pi
i\left\{ 1,(C+1)\right\} .  \label{firstInt}
\end{equation}%
Moreover, the integral in the left hand of (\ref{firstInt}) takes the value $%
\pm 2\pi i$ if and only if $\gamma $ is such that it can be continuously
deformed in $%
\mathbb{C}
\setminus \{0\}$ to an arc not intersecting the positive real axis.
\end{proposition}

If we denote $\Omega =\left\{ C\in
\mathbb{C}
:\Im C\geq 0\right\} ,$ and $R_{C}(z)=-\frac{D_{C}(z)}{z^{2}},$ then
conditions (\ref{1}),(\ref{2}), and (\ref{3}) are full-filled. Since it can
be easily shown that for $C\in \left( -1,+\infty \right) ,$ the zeros $%
a\left( C\right) $ and $b\left( C\right) $ satisfy
\begin{equation*}
0<b\left( C\right) <a\left( C\right) ,
\end{equation*}%
and the segment $\left[ b\left( C\right) ,a\left( C\right) \right] $ is a
short trajectory of the quadratic differential \ref{laguerre} (see Figure
\ref{Figd}), we conclude the existence of the short trajectory for any $C$ $%
\in \Omega $.
\begin{figure}[h]%
\centering
\fbox{\includegraphics[
height=2.0903in,
width=3.7048in
]%
{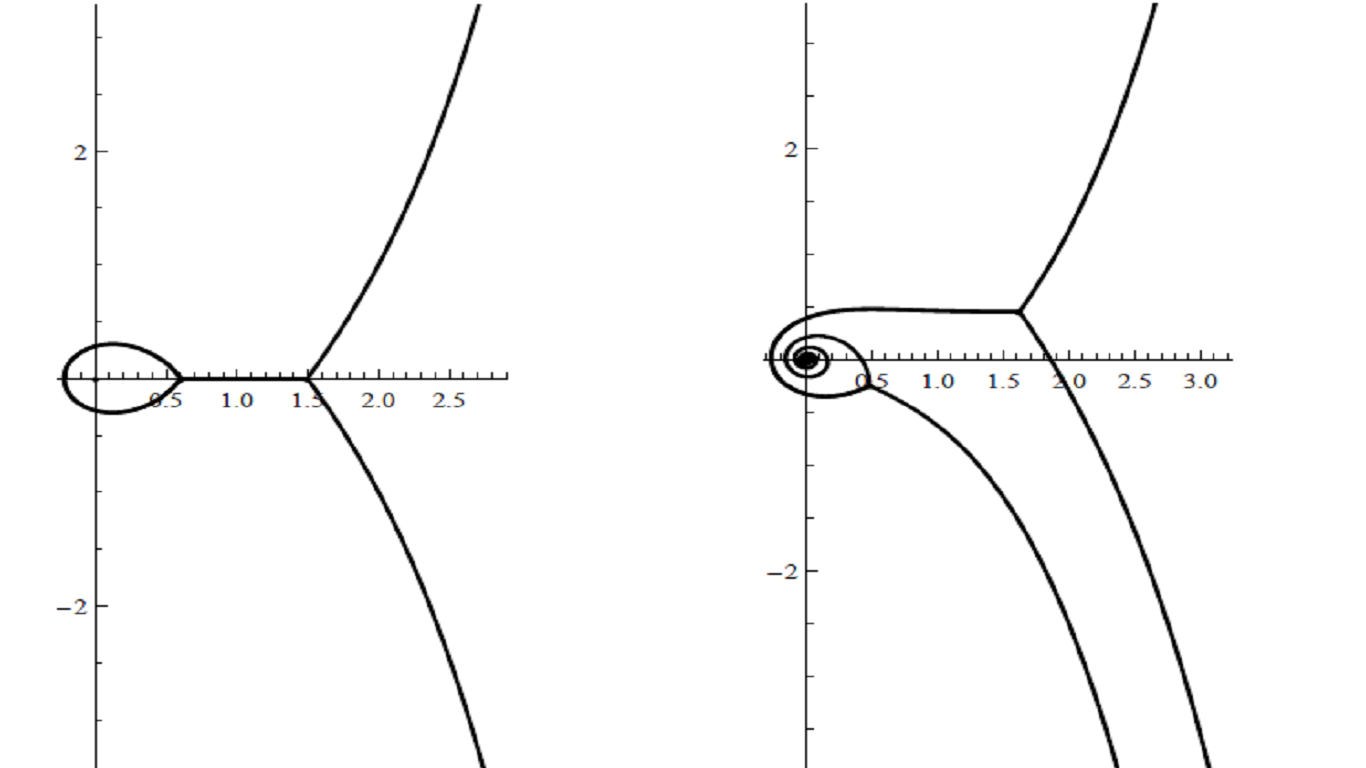}%
}\caption{Critical graphs of $\protect%
\varpi _{-0.95}$ (left) and $\protect\varpi _{-0.95+0.1i}$ (right).}
\label{Figd}%
\end{figure}

\begin{proposition}[\protect\cite{abjk amf},\protect\cite{AMF FT}]
Let $A,B$ satisfy assumptions
\begin{equation}
A+1\neq 0,B+1\neq 0,A+B+1\neq 0,A+B+2\neq 0,  \label{cond AB}
\end{equation}%
and let $\gamma $ be a Jordan arc in $%
\mathbb{C}
\setminus \{-1,1\}$ joining the zeros of $D_{A,B}$, and $\sqrt{D_{A,B}}$ is
its single-valued branch in $%
\mathbb{C}
\setminus \gamma $ fixed by the condition%
\begin{equation*}
\sqrt{D_{A,B}}\left( z\right) \sim \left( A+B+2\right) z,z\rightarrow \infty
.
\end{equation*}%
Then
\begin{equation}
\int_{\gamma }\frac{\left( \sqrt{D_{A,B}\left( t\right) }\right) _{+}}{%
t^{2}-1}dt\in \pm 2\pi i\left\{ 1,\left( A+1\right) ,\left( B+1\right)
,\left( A+B+1\right) \right\} ,  \label{secondInt}
\end{equation}%
where $\left( \sqrt{D_{A,B}\left( t\right) }\right) _{+}$ is the boundary
value on one of the sides of $\gamma $.

Moreover, if in addition of (\ref{cond AB}), $B>0,$ then the integral in the
left hand side of (\ref{secondInt}) takes the value $\pm 2\pi i$ if and only
if $\gamma $ is such that conditions
\begin{equation*}
\sqrt{D_{A,B}}(1)=2A,\quad \sqrt{D_{A,B}}(-1)=-2B
\end{equation*}%
are satisfied.
\end{proposition}

For $B>-1$ we denote
\begin{equation*}
\Omega =\left\{ A\in
\mathbb{C}
:A+1\neq 0,A+B+1\neq 0,A+B+2\neq 0\right\} ,
\end{equation*}%
and $R_{A}(z)=-\frac{D_{A,B}(z)}{\left( z^{2}-1\right) ^{2}},$ then
conditions (\ref{1}),(\ref{2}), and (\ref{3}) are satisfied. Taking into
account that for $A\in
\mathbb{R}
\cap \Omega ,$ there exists a short trajectory of the quadratic differential %
\ref{jacobi}, we conclude the existence of the short
trajectory for any $A$ $\in \Omega $. By repeating the same reasoning, we
conclude the result for any $A$ and $B$ satisfying (\ref{cond AB}) (see
Figures \ref{Fige}, \ref{Figf}).
\begin{figure}[h]
\centering
\begin{minipage}[t]{9cm}
\centering
\fbox{\includegraphics[
width=3.7048in
]
{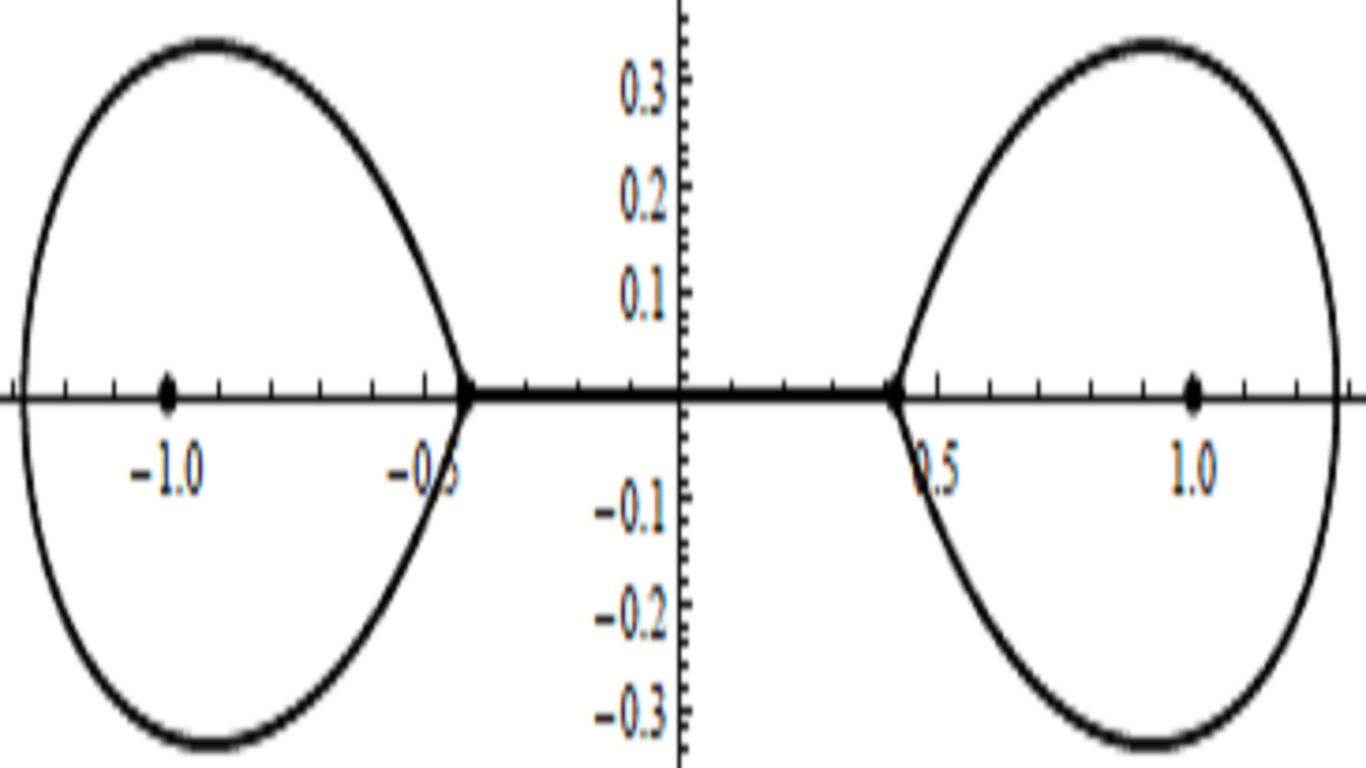}
}\caption{Critical graph of $\protect\varpi _{10,10}.$.}
\label{Fige}
\end{minipage}
\begin{minipage}[t]{9cm}
\centering
\fbox{\includegraphics[
width=3.7048in
]%
{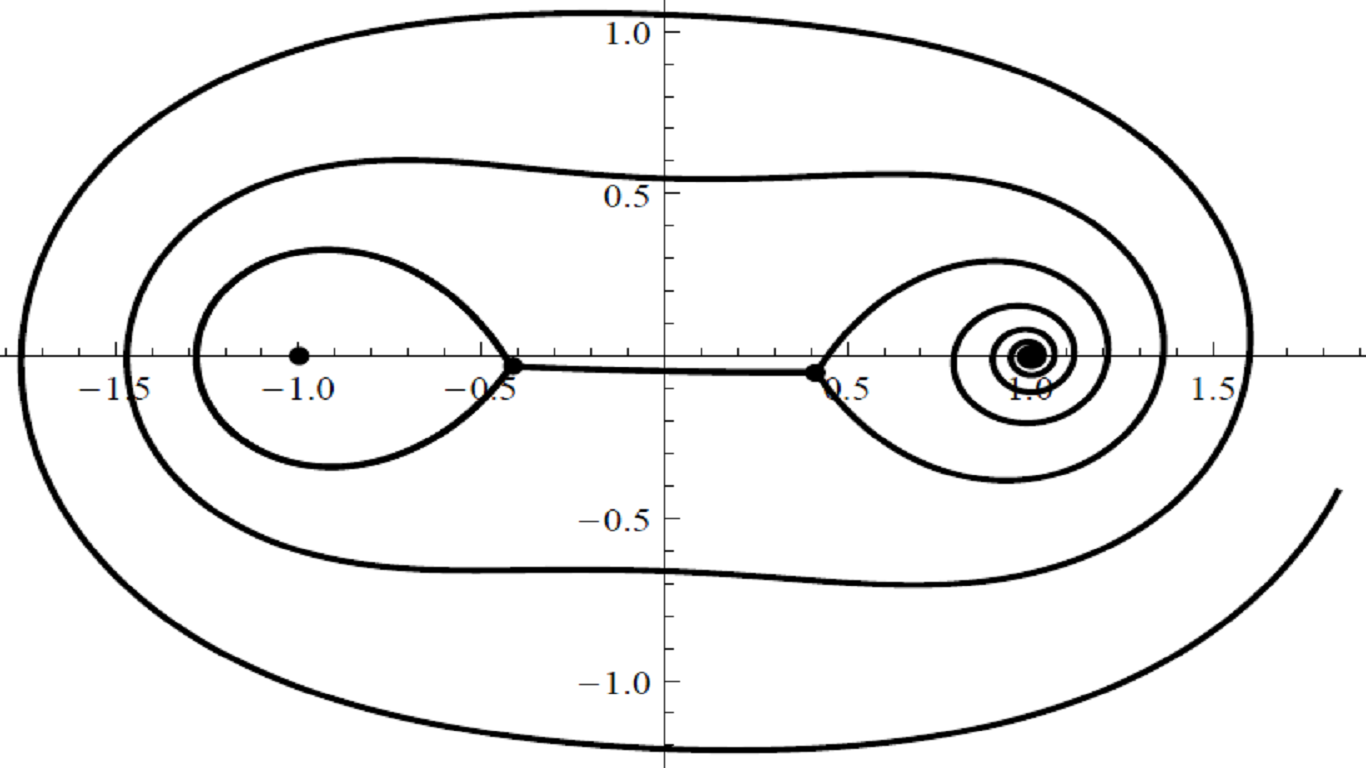}
}\caption{Critical graph of $
\varpi _{10+i,10}$}
\label{Figf}%
\end{minipage}
\end{figure}

\begin{acknowledgement}
This note is a resume of discussions with Professor Andrei Mart\'{\i}%
nez-Finkelshtein; it was carried out during a visit to the Department of
Mathematics of the Stockholm University. The author acknowledges the
hospitality of the hosting department, and especially, Professor Boris
Shapiro. This work was entirely supported by the Stockholm university.

The author acknowledges the contribution of the anonymous referee whose
careful reading of the manuscript helped to improve the presentation.
\end{acknowledgement}

\texttt{Institut Sup\'{e}rieur des Sciences }

\texttt{Appliqu\'{e}es et de Technologie de Gab\`{e}s, }

\texttt{Avenue Omar Ibn El Khattab, 6029. Tunisia.}

\texttt{E-mail adress:faouzithabet@yahoo.fr}

\end{document}